\documentclass[11pt]{article}
\usepackage{geometry}                		
\usepackage[parfill]{parskip}    		
\usepackage{amssymb}
\usepackage{amsmath}

\usepackage{palatino}
\usepackage{textcomp}
\usepackage{graphicx}
\usepackage{epsfig}
\usepackage{epstopdf}
\usepackage{mathrsfs}
\usepackage[english]{babel}
\usepackage{amscd}
\usepackage{amsthm}

\usepackage{fancyhdr}
\newtheorem{theorem}{Theorem}[section]
\newtheorem{lemma}[theorem]{Lemma}
\newtheorem{proposition}[theorem]{Proposition}

\newtheorem{definition}{Definition}

\newtheorem{example}[theorem]{Example}
\newtheorem{problem}[theorem]{Problem}

\title{H\"older-type inequalities and their applications to concentration and correlation bounds}

\author{Christos Pelekis\thanks{Department of Computer Science, KU Leuven,
Celestijnenlaan 200A,
3001, Belgium, Email: pelekis.chr@gmail.com}  \quad \quad
Jan Ramon\thanks{Department of Computer Science, KU Leuven,
Celestijnenlaan 200A,
3001, Belgium, Email: Jan.Ramon@cs.kuleuven.be}  \quad \quad Yuyi Wang\thanks{Department of Computer Science, KU Leuven,
Celestijnenlaan 200A,
3001, Belgium, Email: Yuyi.Wang@cs.kuleuven.be} }

\begin{document}

\maketitle

\begin{abstract} Let $Y_v, v\in V,$ be $[0,1]$-valued random variables having a dependency graph
$G=(V,E)$. We show that
\[   \mathbb{E}\left[\prod_{v\in V} Y_{v} \right]  \leq  
\prod_{v\in V} \left\{ \mathbb{E}\left[Y_v^{\frac{\chi_b}{b}}\right] \right\}^{\frac{b}{\chi_b}}, \]
where $\chi_b$ is the $b$-fold chromatic number of  $G$. This inequality may be seen as a
dependency-graph analogue of a generalised H\"older inequality, due to
Helmut Finner. 
Additionally,  
we provide applications of H\"older-type inequalities to 
concentration and correlation bounds for sums of weakly dependent random variables.
\end{abstract}

\noindent {\emph{Keywords}:  fractional chromatic number; Finner's inequality; Janson's inequality;
dependency graph; hypergraphs

\section{Introduction}\label{intro}

The main purpose of this article is to illustrate that certain 
H\"older-type inequalities can be employed in order to obtain 
concentration and correlation bounds for sums of, possibly dependent,
real-valued random variables whose
dependencies are described in terms of graphs, or hypergraphs.
Before being more precise, let us begin with some notation and definitions that will be fixed
throughout the text. \\
A \emph{hypergraph} $\mathcal{H}$ is a pair $(V,\mathcal{E})$ where $V$ is a finite set
and $\mathcal{E}$ is a family of subsets of $V$. The set $V$ is called the \emph{vertex set}
of $\mathcal{H}$ and the set $\mathcal{E}$ is called the \emph{edge set} of $\mathcal{H}$;
the elements of $\mathcal{E}$ are called \emph{hyperedges} or just edges.
The cardinality of the vertex set will be denoted by $|V|$ and the cardinality of the edge set
by $|\mathcal{E}|$.
A hypergraph is called $k$-\emph{uniform} if every edge from $\mathcal{E}$ has cardinality $k$.
A $2$-uniform
hypergraph is a \emph{graph}. The \emph{degree} of a vertex $v\in V$ is defined as the
number of edges that contain $v$. A hypergraph will be called $d$-regular if every vertex has degree $d$.
A subset $V'\subseteq V$ is called \emph{independent} if it does not contain any edge from $\mathcal{E}$.
A \emph{fractional matching} of a hypergraph, $\mathcal{H}=(V,\mathcal{E})$, is a function
$\phi:\mathcal{E}\rightarrow [0,1]$ such that $\sum_{e:v\in e}\phi(e)\leq 1$, holds true for all
vertices $v\in V$. The \emph{fractional matching number} of $\mathcal{H}$, denoted $\nu^*(\mathcal{H})$,
is defined as $\max_{\phi} \sum_{e\in \mathcal{E}} \phi(e)$ where the maximum runs over
all fractional matchings of $\mathcal{H}$.
The \emph{chromatic number} of a graph $G$ is defined in the following way.
A $b$-\emph{fold coloring} of $G$ is an  assignment of sets of size $b$ to the vertices of the graph
in such a way that adjacent vertices have disjoint sets. A graph is $(a:b)$-\emph{colorable} if it has
a $b$-fold coloring using $a$ different colors. The least $a$ for which the graph is $(a:b)$-colorable
is the $b$-\emph{fold chromatic number} of the graph, denoted $\chi_b(G)$. The \emph{fractional chromatic number}
of a graph $G$ is defined as $\chi^*(G)=\inf_b \frac{\chi_b(G)}{b}$.
Here and later, $\mathbb{P}[\cdot]$ and
$\mathbb{E}[\cdot]$ will denote probability and expectation, respectively. \\

Let us also recall H\"older's inequality. Let $(\Omega,\mathcal{A}, \mathbb{P})$
be a probability space. Let $A$ be a finite set and let
$Y_a,a\in A,$ be random variables from $\Omega$ into $\mathbb{R}$.  Suppose that
$w_a,a\in A$ are non-negative weights such that $\sum_{a\in A} w_a\leq 1$ and each $Y_a$ has
finite $\frac{1}{w_a}$-moment, i.e., $\mathbb{E}\left[Y_a^{1/w_a}\right]<+\infty$, for all $a\in A$.
H\"older's inequality asserts that
\[ \mathbb{E}\left[\prod_{a\in A}Y_a\right] \leq \prod_{a\in A} \mathbb{E}\left[Y_a^{1/w_a}\right]^{w_a}  . \]
This is a classic result (see \cite{Bogachev}).
In this article we shall be interested in applications of H\"older-type inequalities to concentration and
correlation bounds for sums of weakly dependent random variables. We focus on two particular
types of dependencies between random variables.
The first  one is described in terms of a hypergraph. \\

\begin{definition}[hypergraph-correlated random variables]\label{dephypergraph}
Let $\mathcal{H}=(V,\mathcal{E})$ be a hypergraph.
Suppose that $\{Y_e\}_{e\in \mathcal{E}}$ is a collection of real-valued
random variables, indexed by the edge set
of $\mathcal{H}$, that satisfy the following:
there exist \emph{independent} random variables $\{X_v\}_{v\in V}$  indexed by the vertex set $V$ such
that, for every edge $e\in \mathcal{E}$,   $Y_e = f_e(X_v;v\in e)$ is a
function that depends only on the random
variables $X_v$ with $v\in e$. We will refer to the  aforementioned random variables $\{Y_e\}_{e\in \mathcal{E}}$ as
$\mathcal{H}$-\emph{correlated}, or simply as \emph{hypergraph-correlated}, when there is no
confusion about the underlying hypergraph.
\end{definition}

Hypergraph-correlated random variables are encountered in the theory of random
graphs (see \cite{Gavinsky, Jansonthree, Wolfovitz} and references therein). 
Another type of ``dependency structure`` that plays a key role in 
probabilistic combinatorics and related areas involves the notion of dependency graphs (see \cite{Alon, Jansontwo}).  \\

\begin{definition}[Dependency graph]\label{depgraph}
A \emph{dependency graph} for the random variables $\{Y_v\}_{v\in V}$, indexed by a finite set $V$, is any loopless graph, $G=(V,E)$,
whose vertex set $V$ is the index set of the random variables and whose edge set
is such that if $V'\subseteq V$ and $v_i\in V$ is not incident to any vertex of $V'$, then  $Y_v$ is
mutually independent of the random variables $Y_{v'}$ for which $v' \in V'$.  We will refer
to  random variables $\{Y_v\}_v$ having a dependency graph $G$ as $G$-\emph{dependent} or
as \emph{graph-dependent}.
\end{definition}

If $\{Y_e\}_{e\in\mathcal{E}}$ are hypergraph-correlated random variables, then one can
define their dependency graph whose vertex set is $\mathcal{E}$ and with edges joining any two
sets $e,e'\in \mathcal{E}$ such that $e\cap e' \neq \emptyset$.
Hence a set of hypergraph-correlated random variables is graph-dependent. The reader might
wonder whether the converse holds true. We will see, using a particular
generalisation of H\"older's inequality,  that this is not the case (see Example \ref{examexam} below) and so the
aforementioned notions of dependencies are not equivalent.  \\
In the present paper we shall be interested in employing
H\"older-type inequalities in order to obtain concentration and
correlation bounds for sums of hypergraph-correlated  random variables
as well as for sums of graph-dependent
random variables. \\
The main results are stated in Section \ref{sectres} and the proofs are
contained in Sections  \ref{proofgraphFinner}
and \ref{Finnerappl}.

\section{Results}\label{sectres}
\subsection{H\"older-type inequalities}

We begin with the
following theorem, due to Helmut Finner, that provides
a generalisation of H\"older's inequality for hypergraph-correlated random variables.   \\

\begin{theorem}[Finner \cite{Finner}]
\label{mainthm} Let $\mathcal{H}=(V,\mathcal{E})$ be a hypergraph and let $\{Y_e\}_{e\in \mathcal{E}}$
be $\mathcal{H}$-correlated random variables.
If $\phi:\mathcal{E}\rightarrow [0,1]$ is a fractional matching of $\mathcal{H}$ then
\[\mathbb{E}\left[\prod_{e\in \mathcal{E}}Y_e\right] \leq \prod_{e\in \mathcal{E}}\left\{\mathbb{E}\left[Y_e^{1/\phi(e)}\right]\right\}^{\phi(e)} .\]
\end{theorem}

Notice that, by applying the previous result to the random variables $Z_e =Y_{e}^{\phi(e)}$, one concludes
$\mathbb{E}\left[\prod_{e\in \mathcal{E}}Y_e^{\phi(e)}\right] \leq \prod_{e\in \mathcal{E}}\left\{\mathbb{E}\left[Y_e\right]\right\}^{\phi(e)}$.
See \cite{Finner} for a proof of this result that is based on Fubini's theorem and H\"older's inequality.
Alternatively, see \cite{Ramon} for a proof that uses the concavity of the weighted geometric mean and
Jensen's inequality.
In other words, the previous result provides a H\"older-type inequality for
hypergraph-correlated  random variables which is formalised in terms of a fractional matching of
the underlying hypergraph.
In this article we provide an
analogue of Theorem \ref{mainthm} for random variables having a depenency graph $G$. The
corresponding H\"older-type inequality is formalised in terms of the $b$-fold chromatic number of
$G$. More precisely,
we obtain the following result. \\

\begin{theorem}\label{graphFinner} Let  $\{Y_v\}_{v}$ be real-valued random variables having a dependency
graph $G=(V,E)$.
Then, for every $b$-fold coloring of $G$ using $\chi_b:=\chi_b(G)$ colors, we have
\[ \mathbb{E}\left[\prod_{v\in V} Y_{v} \right]  \leq\prod_{v\in V} \left\{ \mathbb{E}\left[Y_{v}^{\frac{\chi_b}{b}}\right] \right\}^{\frac{b}{\chi_b}}. \]
\end{theorem}

We prove Theorem \ref{graphFinner} in
Section \ref{proofgraphFinner}. The proof employs
the concavity of the weighted geometric mean and the definition of $b$-fold chromatic number.
In the remaining part of the current section we discuss applications of Theorem \ref{mainthm} and Theorem
\ref{graphFinner} to concentration and correlation bounds for sums of hypergraph-correlated
random variables as well as for sums of random variables having a dependency graph.
We begin with the later case.

\subsection{Applications}
\subsubsection{Dependency graphs}

Theorem \ref{graphFinner}, combined with standard techniques based on
exponential moments,
yields concentration inequalities for sums of random variables having a  dependency graph.
More precisely,
Theorem \ref{graphFinner} yields a new proof of  the following estimate on the probability
that the sum of graph-dependent random variables is significantly larger than its mean.\\

\begin{theorem}[Janson \cite{Jansonthree}]\label{chromaticthm}
Let  $\{Y_v\}_{v\in V}$ be $[0,1]$-valued random variables having a dependency
graph $G=(V,E)$.
Set $q:=\frac{1}{|V|} \mathbb{E}\left[\sum_v Y_v\right]$.  If
$t=n(q+\varepsilon)$ for some $\varepsilon>0$, then
\[ \mathbb{P}\left[\sum_v Y_v \geq t\right] \leq \text{exp}\left(  -\frac{2\varepsilon^2 |V|}{\chi^*}\right), \]
where $\chi^*=\chi^*(G_{\mathcal{H}})$ is the fractional chromatic number of $G$.
\end{theorem}

See \cite{Jansonthree}, Theorem $2.1$, for a proof of this result that is based on breaking up
the sum into a particular linear combination of sums of independent random variables.
In Section \ref{proofgraphFinner}
we provide a new proof of Theorem \ref{chromaticthm}
which is based on Theorem \ref{graphFinner}. Moreover,
under additional information on the variance of the random variables, we obtain the following
Bennett-type inequality. \\

\begin{theorem}\label{variancebound} Let  $\{Y_v\}_{v\in V}$ be random variables having a dependency
graph $G=(V,E)$. For every $v\in V$ let $\sigma_v^2 := \text{Var}\left(Y_v\right)$
and assume further that $Y_v\leq 1$ and $\mathbb{E}[Y_v]=0$.
Set $S=\sum_v \sigma_{v}^{2}$ and
fix $t>0$. Then
\[ \mathbb{P}\left[\sum_v Y_v \geq t \right] \leq  \text{exp}\left(-\frac{S}{\chi^{\ast}(G)} \psi\left(\frac{t}{S}\right) \right) ,  \]
where $\psi(x) = (1+x)\ln(1+x)- x$.
\end{theorem}

Let us remark that the previous result is in fact
an improvement upon Theorem $2.3$ from \cite{Jansonthree}. Indeed, in \cite{Jansonthree} Theorem $2.3$, the 
bound  $\text{exp}\left(-\frac{S}{\chi^{\ast}(G)} \psi\left(\frac{4t}{5S}\right)\right)$ is obtained 
on the tail probability of Theorem \ref{variancebound}. 
Notice that we assume a one-sided bound on each $Y_v$. The proof of the previous result is similar to
the proof of Theorem \ref{chromaticthm}; we sketch it in Section \ref{proofgraphFinner}.
In the next section we discuss application of Theorem \ref{mainthm} to
sums of hypergraph-correlated random variables.

\subsubsection{Hypergraph-correlated random variables}

In this section we discuss applications of Finner's inequality. 
We begin by applying Theorem \ref{mainthm} to the following question. \\

\begin{problem}\label{mainproblem}
Fix a hypergraph $\mathcal{H}=(V,\mathcal{E})$
and let $\mathbb{I}$ be a random subset of $V$ formed by including vertex $v\in V$ 
in $\mathbb{I}$ with probability $p_v \in (0,1)$, independently of other vertices.
What is an upper bound on the probability that $\mathbb{I}$ is independent?
\end{problem}

Here and later, given a set of parameters in $(0,1)$, say $\mathbf{p}=\{p_v\}_{v\in V}$, 
indexed by the vertex set of a hypergraph, 
we will denote by $\pi(\mathbf{p}, \mathcal{H})$ the probability that $\mathbb{I}$ is independent. 
The previous problem has attracted the attention
of several authors and appears to be related to a variety of topics
(see \cite{Boppana, Galvin, Jansonfour, Krivelevich, Wolfovitz} and references therein).
A particular line of research is motivated by question about independent sets
and subgraph counting in random graphs.
In this context,
Problem \ref{mainproblem} has been considered by Janson et al. \cite{Jansonfour},  Krivelevich et al. \cite{Krivelevich}
and Wolfovitz \cite{Wolfovitz}.
It is observed in \cite{Krivelevich} that when $\mathcal{H}$ is $k$-uniform and $d$-regular an
exponential estimate on $\pi(\mathbf{p},\mathcal{H})$, can be obtained using the so-called Janson's
inequality (see \cite{Jansontwo}, Chapter $2$). Additionally, it is shown that under certain "mild additional
assumptions" the bound provided by Janson's inequality can be improved to
\[ \pi(\mathbf{p},\mathcal{H}) \leq \text{exp}\left(-\Omega\left( \frac{p|\mathcal{E}|}{(1-p)k d}\right) \right) . \]
See \cite{Krivelevich} and for a precise formulation of the additional assumptions and a proof of this
result that is based on a martingale-type concentration inequality.
In Section \ref{Finnerappl} we provide the following upper bound
on $\pi(\mathbf{p},\mathcal{H})$ using Finner's inequality.  \\

\begin{theorem}
\label{finnerbound} Let $\mathcal{H}=(V,\mathcal{E})$ be a hypergraph.  
For each $e\in \mathcal{E}$, let $|e|$ denote its cardinality. 
Then
\[  \pi(\mathbf{p},\mathcal{H}) \leq \prod_e \left(1-\prod_{v\in e} p_v\right)^{\phi(e)}, \] 
where $\phi: \mathcal{E} \rightarrow [0,1]$ is a fractional matching of $\mathcal{H}$. 
In particular, if the hypergraph $\mathcal{H}$ is $k$-uniform and $p_v=p$, for all $v\in V$ then 
\[  \pi(\mathbf{p},\mathcal{H}) \leq  \left(1-p^k\right)^{\nu^*(\mathcal{H})}, \]
where $\nu^*(\mathcal{H})$ is the fractional matching number of $\mathcal{H}$.
\end{theorem}

Let us remark that the second statement in Theorem \ref{finnerbound} has a \emph{monotonicity property}, in 
the sense that  
if $\mathcal{H}_1$ is a superhypergraph of $\mathcal{H}_2$ then 
$\left(1-p^k\right)^{\nu^*(\mathcal{H}_1)} \leq \left(1-p^k\right)^{\nu^*(\mathcal{H}_2)}$.

In Section \ref{Finnerappl} we show that Theorem \ref{finnerbound} 
can be seen as an alternative to Janson's inequality.
Moreover, using Finner's inequality, one can conclude that the two notions of
dependencies given in Definition \ref{dephypergraph} and Definition \ref{depgraph}
are not equivalent. \\

\begin{example}\label{examexam} Let $G$ be a cycle-graph on $5$ vertices $\{v_1,\ldots,v_5\}$ such that
$v_i$ is incident to $v_{i+1}$, for $i\in \{1,2,3,4\}$ and $v_5$ is incident to $v_1$. Let
$Y=(Y_1,\ldots,Y_5)$ be a vector of Bernoulli $0/1$ random variables whose  distribution
is defined as follows. The vector $Y$ takes the value $(0,0,0,0,0)$ with probability $\frac{1}{2}(2-p)(1-p)^2$,
the value $(1,1,1,1,1)$ with probability $\frac{p^2+p^3}{2}$, the values
\[ (0,0,0,1,1), (0,0,1,1,0), (0,1,1,0,0), (1,1,0,0,0), (1,0,0,0,1)\]
with probability $\frac{p(1-p)^2}{2}$,
the values
\[ (0,0,1,1,1), (0,1,1,1,0), (1,1,1,0,0), (1,1,0,0,1), (1,0,0,1,1)\]
with probability $\frac{p^2-p^3}{2}$
and the remaining values with probability $0$.  Elementary, though quite tedious, calculations show that
$\mathbb{E}[Y_j]=p$, for $j=1,\ldots,5$ and that $G$ is a dependency graph for $\{Y_j\}_{j=1}^{5}$.
Now assume that $\{Y_j\}_{j=1}^{5}$ are $\mathcal{H}$-correlated, 
for some hypergraph $\mathcal{H}=(V,\mathcal{E})$. 
Notice that $|\mathcal{E}|=5$.  
If $e_i \in \mathcal{E}$ 
is the edge corresponding to the random variable $Y_i, i=1,\ldots,5,$ then the fact that 
$\{Y_i\}_{i=1}^{5}$ have $G$ as a depencency graph implies that 
$e_i \cap e_{(i+2)\mod 5} = \emptyset$, 
for 
$i\in \{1,2,3,4,5\}$.
This means that the fractional matching number of $\mathcal{H}$ is at least $2.5$ and therefore 
Theorem \ref{mainthm}
implies that $\mathbb{P}\left[Y=(1,1,1,1,1)\right] \leq p^{2.5}$. However, the arithmetic
geometric means inequality implies $\frac{p^2+p^3}{2} > p^{2.5}$ and therefore the random variables
$\{Y_j\}_{j=1}^{5}$ are not hypergraph-correlated.
\end{example}

In the same vein as in the previous section, 
Theorem \ref{mainthm} can be employed in order to deduce 
concentration inequalities for sums of hypergraph-correlated random variables.
This has been reported in prior work and so we only provide
the statement without proof. In \cite{Ramon} one can find a proof of
the following result. \\

\begin{theorem}[Ramon et al. \cite{Ramon}] Let $\mathcal{H}= (V,\mathcal{E})$ be a
hypergraph and assume that $\{Y_e\}_{e\in \mathcal{E}}$ are $\mathcal{H}$-correlated random variables.
Assume further that $Y_e\in [0,1]$, for all $e\in \mathcal{E}$, and that
$\mathbb{E}\left[Y_{e}\right] =p_e$, for some $p_e\in (0,1)$.
Let $\phi:\mathcal{E}\rightarrow [0,1]$ be a fractional matching
of $\mathcal{H}$ and set $\Phi = \sum_{e} \phi(e), \; p=\frac{1}{|\mathcal{E}|}\sum_e \mathbb{E}\left[Y_e\right]$. 
If $t$ is a real number from the interval $\left(\Phi p, \Phi\right)$ such that
$t=\Phi(p+\varepsilon)$, then
\[ \mathbb{P}\left[\sum_{e\in \mathcal{E}} \phi(e)Y_e \geq t \right] \leq \text{exp}\left(-2\Phi \varepsilon^2\right) . \]
\end{theorem}

In particular, if $d$ is the maximum degree of $\mathcal{H}$ and $\phi(e)=\frac{1}{d}$, for all
$e\in \mathcal{E}$, then the previous result yields the bound
\[ \mathbb{P}\left[\sum_e Y_e \geq t \right] \leq \text{exp}\left(-2\frac{|\mathcal{E}|}{d} \varepsilon^2\right), \; 
\text{for} \; t= |\mathcal{E}|(p +\varepsilon)  . \]
This inequality has also been obtained in Gavinsky et al. \cite{Gavinsky} using entropy ideas.

\section{Proofs - dependency graphs}
\label{proofgraphFinner}

In this section we prove Theorem \ref{graphFinner}, Theorem \ref{chromaticthm}
and Theorem \ref{variancebound}.
The proof of the first theorem will require the concavity of the weighted geometric mean. \\

\begin{lemma}\label{concave} Let $\beta= (\beta_1,\ldots,\beta_k)$ be a vector of non-negative real numbers
such that $\sum_{i=1}^{k} \beta_i = 1$. Then the function $g:\mathbb{R}^k\rightarrow \mathbb{R}$ defined
by $g(t) = \prod_{i=1}^{k} t_{i}^{\beta_i}$ is concave.
\end{lemma}
\begin{proof} This is easily verified by showing that the Hessian matrix is positive definite. See
\cite{Boyd}, or \cite{Ramon}
for details.
\end{proof}

We now ready to prove Theorem \ref{graphFinner}.

\begin{proof}[Proof of Theorem \ref{graphFinner}]
We show that
\[ \mathbb{E}\left[\prod_{v\in V} \{Y_{v}\}^{\frac{b}{\chi_b}} \right]  \leq\prod_{v\in V} \left\{ \mathbb{E}\left[Y_{v}\right] \right\}^{\frac{b}{\chi_b}}. \]
The theorem follows by applying this inequality to the random variables $Z_v=Y_{v}^{\frac{\chi_b}{b}}$.
For every color $i=1,\ldots,\chi_b$ let $I_i$ be the set consisting of the vertices that are
colored with color $i$. Note that each $I_i$ is an independent subset of $V$ and every vertex $v\in V$
appears in exactly $b$ independent sets $I_i$. Therefore,
\[ \mathbb{E}\left[\prod_{v\in V} \left\{ Y_{v}\right\}^{\frac{b}{\chi_b}} \right]  =
\mathbb{E}\left[\prod_{i=1}^{\chi_b} \prod_{v\in I_i}  \left\{Y_{v}\right\}^{\frac{1}{\chi_b}}\right]
= \mathbb{E}\left[ \prod_{i=1}^{\chi_b}\left\{\prod_{v\in I_i}Y_v\right\}^{\frac{1}{\chi_b}} \right] . \]
Lemma \ref{concave} and Jensen's inequality combined with
the observation that the random variables $\{Y_v\}_{v\in I_i}$ are mutually independent yield
\[  \mathbb{E}\left[ \prod_{i=1}^{\chi_b}\left\{\prod_{v\in I_i}Y_v\right\}^{\frac{1}{\chi_b}} \right] \leq
\prod_{i=1}^{\chi_b} \left\{\mathbb{E}\left[\prod_{v\in I_i} Y_v\right]\right\}^{\frac{1}{\chi_b}} =
\prod_{i=1}^{\chi_b}  \prod_{v\in I_i} \left\{  \mathbb{E}\left[Y_v\right]\right\}^{\frac{1}{\chi_b}} .
\]
Now, using again the fact that each vertex $v$ appears in exactly $b$ sets $I_i$, we  conclude
\[ \prod_{i=1}^{\chi_b}  \prod_{v\in I_i} \left\{  \mathbb{E}\left[Y_v\right]\right\}^{\frac{1}{\chi_b}} =
\prod_{v\in V} \left\{\mathbb{E}\left[Y_v\right]\right\}^{\frac{b}{\chi_b}}
\]
and the result follows.
\end{proof}

Theorem \ref{graphFinner} yields a new proof of Theorem \ref{chromaticthm}.
Let us first recall the following, well-known, result whose proof is included for the sake of completeness. \\

\begin{lemma}\label{convorder} Let $X$ be a random variable that takes values on the interval $[0,1]$.
Suppose that $\mathbb{E}[X]=p$, for some $p\in (0,1)$, and let $B$ be a Bernoulli $0/1$ random
variable such that $\mathbb{E}[B]=p$. If $f:[0,1]\rightarrow \mathbb{R}$ is a convex function, then
$\mathbb{E}\left[f(X)\right] \leq \mathbb{E}\left[f(B)\right]$.
\end{lemma}
\begin{proof} Given an outcome from the random variable $X$, define the random variable $B_X$
that takes the values $0$ and $1$ with probability $1-X$ and $X$, respectively. It is easy to see
that $\mathbb{E}\left[B_X\right]=p$ and so $B_X$ has the same
distribution as $B$. Now Jensen's inequality implies
\[ \mathbb{E}\left[f(X)\right] = \mathbb{E}\left[f\left(\mathbb{E}\left[B_X\big| X\right]\right)\right] \leq \mathbb{E}\left[f\left(B_X \big| X\right)\right] = \mathbb{E}\left[f(B_X)\right] , \]
as required.
\end{proof}

We now proceed with the proof of Theorem \ref{chromaticthm}.

\begin{proof}[Proof of Theorem \ref{chromaticthm}] Fix $h>0$ and let $q_v = \mathbb{E}\left[Y_v\right]$,
for $v\in V$.
Using Markov's inequality and
Theorem \ref{graphFinner} we estimate
\begin{eqnarray*}
\mathbb{P}\left[\sum_{v\in V} Y_v \geq t \right] &\leq& e^{-ht}\mathbb{E}\left[e^{h\sum_{v\in V} Y_v} \right] \\
&=& e^{-ht}\mathbb{E}\left[\prod_{v\in V}e^{h Y_v} \right] \\
&\leq& e^{-ht}  \prod_{v\in V}\left\{\mathbb{E}\left[\text{exp}\left(\frac{\chi_b}{b}h Y_v\right)\right]\right\}^{\frac{b}{\chi_b}}
\end{eqnarray*}
For $v\in V$ let $B_v$ be a Bernoulli $0/1$ random variable of mean $q_v$. The previous lemma
implies
\begin{eqnarray*}
e^{-ht}  \prod_{v\in V}\left\{\mathbb{E}\left[\text{exp}\left(\frac{\chi_b}{b}h Y_v\right)\right]\right\}^{\frac{b}{\chi_b}}     &\leq& e^{-ht}  \prod_{v\in V}\left\{\mathbb{E}\left[\text{exp}\left(\frac{\chi_b}{b}h B_v\right)\right]\right\}^{\frac{b}{\chi_b}} \\
&=& e^{-ht}  \prod_{v\in V} \left\{(1-q_v)  + q_v e^{\frac{\chi_b}{b}h} \right\}^{\frac{b}{\chi_b}}.
\end{eqnarray*}
Using the weighted arithmetic-geometric means inequality we conclude
\begin{eqnarray*} e^{-ht}  \prod_{v\in V} \left\{(1-q_v)  + q_v e^{\frac{\chi_b}{b}h} \right\}^{\frac{b}{\chi_b}}&\leq&
e^{-ht}  \left\{\sum_{v\in V}\frac{1}{|V|}\left((1-q_v) +q_v e^{\frac{\chi_b}{b}h}\right)   \right\}^{\frac{b}{\chi_b}|V|}\\
&=& e^{-ht} \left\{1-q + q e^{\frac{\chi_b}{b}h}\right\}^{\frac{b}{\chi_b}|V|} .
\end{eqnarray*}
If we minimise the last expression with respect to $h>0$ we get that
$h$ must satisfy $e^{\frac{\chi_b}{b}h} = \frac{t(1-q)}{q(|V|-t)}$ and therefore,
since $t=|V|(q+\varepsilon)$, we conclude
\[ \mathbb{P}\left[\sum_{v\in V} Y_v \geq t \right] \leq \left\{\left(\frac{q}{q+\varepsilon} \right)^{q+\varepsilon}
\left(\frac{1-q}{1-(q+\varepsilon)}\right)^{1-(q+\varepsilon)}\right\}^{\frac{b}{\chi_b}|V|}  =
e^{ -\frac{b}{\chi_b}|V| D(q+\varepsilon||q)},  \]
where $D(q+\varepsilon||q)$ is the Kullback-Leibler distance between $q+\varepsilon$ and $q$.
Finally, using the standard estimate $D(q+\varepsilon||q)\geq 2\varepsilon^2$, we deduce
\[ \mathbb{P}\left[\sum_{v\in V} Y_v \geq t \right] \leq e^{-\frac{b}{\chi_b}|V|2\varepsilon^2} \]
and the result follows upon minimising the last expression with respect to $b$.
\end{proof}

The proof of Theorem \ref{variancebound} is similar.

\begin{proof}[Proof of Theorem \ref{variancebound}]
Fix $h>0$ to be determined later.
As in the proof of Theorem \ref{chromaticthm}, Markov's inequality and Theorem
\ref{graphFinner} yield
\begin{eqnarray*} \mathbb{P}\left[\sum_v Y_v \geq t\right] \leq
e^{-ht}  \prod_{v\in V}\left\{\mathbb{E}\left[\text{exp}\left(\frac{\chi_b}{b}h Y_v\right)\right]\right\}^{\frac{b}{\chi_b}} .
\end{eqnarray*}

Using an inequality proved in \cite{Jansonthree} (Inequality (3.7) on page 240), we have
\[\mathbb{E}\left[\exp\left(\frac{\chi_b}{b} h Y_v\right)\right] \le \exp\left(\sigma_v^2 g\left(\frac{\chi_b}{b} h\right)\right) ,\]
where $g(a):=e^a-1-a.$
Summarising, we have shown
\[ \mathbb{P}\left[\sum_v Y_v \geq t\right] \leq \text{exp}\left(-ht + \left(\frac{b}{\chi_b}e^{h\chi_b/b}-\frac{b}{\chi_b}-h \right)S\right).  \]
Now choose $h=\frac{b}{\chi_b}\cdot \ln\left(1+\frac{t}{S}\right)$
to deduce
\begin{eqnarray*} \mathbb{P}\left[\sum_v Y_v \geq t\right] &\leq& \text{exp}\left(\frac{b}{\chi_b}t -S \frac{b}{\chi_b}\left(1+\frac{t}{S}\right)\ln\left(1+\frac{t}{S}\right) \right) \\
&=& \text{exp}\left(-S \frac{b}{\chi_b} \psi\left(\frac{t}{S}\right) \right) .
\end{eqnarray*}
The result follows upon minimising the last expression with respect to $b$.
\end{proof}

\section{Applications of Finner's inequality}\label{Finnerappl}

\subsection{Proof of Theorem \ref{finnerbound}}

In this section we prove Theorem \ref{finnerbound} and discuss  applications of this result
to the theory of random graphs.

\begin{proof}[Proof of Theorem \ref{finnerbound}]
Let $X_v,v\in V,$ be indicators of the event $v\in \mathbb{I}$.
For each $e\in \mathcal{E}$ set $Y_e = \prod_{v\in e} B_v$. Clearly, the
random variables $\{Y_e\}_{e\in \mathcal{E}}$ are $\mathcal{H}$-correlated.
Now look at the probability
$\mathbb{P}\left[\sum_e Y_e =0\right]$. Notice that if all  $Y_e$ are
equal to zero, then every edge $e\in \mathcal{E}$ contains a vertex, $v$, such that $B_v=0$ and
vice versa. This implies that if $\sum_e Y_e=0$ then the set of vertices, $v$, for which $B_v=1$ is
an independent subset of $V$ and vice versa.  Therefore
\[ \mathbb{P}\left[\sum_{e\in \mathcal{E}} Y_e =0\right]= \mathbb{E}\left[\prod_{e\in \mathcal{E}}(1-Y_e)\right] 
= \pi(\mathbf{p},\mathcal{H}) . \]
From Theorem \ref{mainthm} we
deduce
\[\mathbb{E}\left[\prod_{e\in \mathcal{E}}(1-Y_e)\right] \leq \mathbb{E}\left[\prod_{e\in \mathcal{E}}(1- Y_e)^{\phi(e)}\right]
= \prod_{e\in \mathcal{E}}\left(\mathbb{E}\left[1-Y_e\right]\right)^{\phi(e)} \] 
and the first statement follows. 
To prove the second statement notice that $\mathbb{E}\left[Y_e\right] = p^k$, for all $e\in\mathcal{E}$, and therefore
\[ \mathbb{E}\left[\prod_{e\in \mathcal{E}}(1-Y_e)\right] \leq \left( 1-p^k \right)^{\sum_e \phi(e)}. \]
The result follows by maximising the 
expression on the right hand side over all fractional matchings of $\mathcal{H}$. 
\end{proof}

\subsection{Finner's inequality as an alternative to Janson's}
\subsubsection{Triangles in random graphs}

In this section we discuss comparisons between Finner's and Janson's inequality.
Janson's inequality (see Janson \cite{Janson} and Janson et al. \cite[Chapter 2]{Jansontwo}) 
is a well known result that provides estimates on the probability that a  a sum of
dependent indicators is equal to zero. It is described in terms of the dependency graph corresponding
to the indicators. 
More precisely, let $\{B_v\}_{v\in V}$ be indicators having a dependency graph $G$. 
Set $\mu = \mathbb{E}\left[\sum_v B_v\right]$ and $\Delta = \sum_{e=\{u,v\}\in G} \mathbb{E}\left[B_u B_v\right]$. 
Janson's inequality asserts that  
\[ \mathbb{P}\left[\sum_v B_v =0\right] \leq \min \left\{ e^{-\mu+\Delta}, 
\text{exp}\left(\frac{\Delta}{1-\max_v \mathbb{E}[B_v]}\right) \prod_v (1-\mathbb{E}[B_v]) \right\} .   \]
Janson's inequality has been proven to be very useful in the study of the Erd\H{o}s-R\'enyi random
graph model, denoted $\mathcal{G}(n,p)$. Recall that such a model generates
a random graph on $n$ labelled vertices by joining pairs of vertices, independently,
with probability $p\in (0,1)$. For $G\in \mathcal{G}(n,p)$ let us denote by $T_G$ the number
of triangles in $G$.
A typical application of Janson's inequality provides the estimate
\[ \mathbb{P}\left[G\in \mathcal{G}(n,p)\; \text{is triangle-free}\right] \leq (1-p^3)^{\binom{n}{3}}\cdot \text{exp}\left(\frac{\Delta}{2(1-p^3)}\right) ,\]
where $\Delta = 6\binom{n}{4}p^5$.
In this section we juxtapose the previous bound with the bound provided by Finner's inequality. \\

\begin{proposition} Let $G\in \mathcal{G}(n,p)$ be an Erd\H{o}s-R\'enyi random
graph and denote by $T_G$ the
number of triangles in $G$. Then
\[ \mathbb{P}\left[T_G=0\right] \leq \left(1-p^3\right)^{\frac{1}{n-2}\binom{n}{3}}. \]
\end{proposition}
\begin{proof} We apply Theorem \ref{finnerbound}. Define a hypergraph $\mathcal{H}=(\mathcal{V},\mathcal{E})$ as follows.
Let $v_i, i=1,\ldots, \binom{n}{2}$, be an enumeration of all (potential)
edges in $G$  and consider $v_1,\ldots,v_{\binom{n}{2}}$ as the vertex set $\mathcal{V}$ of the hypergraph
$\mathcal{H}$.
Let $B_{v_i}, i=1,\ldots, \binom{n}{2}$ be independent Bernoulli $\text{Ber}(p)$ random variables,
corresponding to the edges of $G$, and
let $E_i, i=1\ldots,\binom{n}{3}$, be an enumeration of all triplets of edges in $G$
that
form (potential) triangles in $G$. Define $\mathcal{E}=\{E_1,\ldots, E_{\binom{n}{3}}\}$ to be the
edge set of $\mathcal{H}$ and
let $Z_i$ be the indicator of triangle $E_i$; thus $T_G=\sum_i Z_i$.
Now form a subset $\mathbb{I}$ of $\mathcal{V}$ by picking each vertex, independently,
with probability $p$. Then the probability that $\mathbb{I}$ is independent equals $\mathbb{P}\left[T_G=0\right]$ and, in order to apply Theorem \ref{finnerbound},
we have to find a fractional matching of $\mathcal{H}$. Since
every vertex of $\mathcal{H}$ belongs to $n-2$ edges in $\mathcal{E}=\{E_1,\ldots,E_{\binom{n}{3}}\}$, we
obtain a fractional matching, $\phi(\cdot)$ of $\mathcal{H}$ by setting $\phi(E_i) =\frac{1}{n-2}$, for  $i=1,\ldots,\binom{n}{3}$.
The result follows.
\end{proof}

Notice that the bound obtained from Janson's inequality is smaller than the previous bound for
values of $p$ that are close to $0$, but the previous bound does better for large values of $p$.
Similar estimates can be obtained for the probability that a graph $G\in \mathcal{G}(n,p)$
contains no $k$-clique, for $k\geq 3$. The details are left to the reader.

\subsubsection{Paths of fixed length between two vertices in a random graph}

In this section we discuss one more application of Finner's inequality.
Let $G\in \mathcal{G}(n,p)$
be a random graph on $n$ labelled vertices. Fix two vertices, say $u$ and $v$.
What is an upper bound on the probability that there is no path of length $k$ between $u$ and $v$? \\

A path of length $k$ is  a sequence of edges $\{v_0,v_1\}, \{v_1,v_2\}, \ldots, \{v_{k-2},v_{k-1}\}, \{v_{k-1},v_k\}$ such that
$v_i\neq v_j$.
We assume $k\geq 3$,  otherwise the problem is easy.
Let $\{P_i\}_i$ be an enumeration of all (potential) paths of length $k$ between $u$ and $v$.
Clearly, there are $\binom{n-2}{k-1}\cdot (k-1)!$ such paths. Define the hypergraph $\mathcal{H}=(\mathcal{V},\mathcal{E})$ as follows. The vertices of $\mathcal{H}$ correspond to the (potential) edges of $G$
and the edges of $\mathcal{H}$ correspond to the sets of edges in $G$ that form a path of
length $k$ between $u$ and $v$.
Hence the probability that there is no path of length $k$
between $u$ and $v$ equals $\pi(p,\mathcal{H})$.
In order to apply Theorem \ref{finnerbound} we have to
find a fractional matching of $\mathcal{H}$ and so it is enough to find
an upper bound on the maximum degree of $\mathcal{H}$.  To this end,
fix an edge, $e=\{x,y\}$, in $G$.
In case one of the vertices $x$ or $y$ is equal to either $u$ or $v$, then there are
$\binom{n-3}{k-2}\cdot (k-2)!$ paths of length $k$ from $u$ to $v$ that pass through edge $e$.
If none of the vertices $x,y$ is equal to $u$ or $v$, then we count the paths as follows.
We first create a path, $P_{k-2}$, of length $k-2$ from $u$ to $v$ that does not pass through any of the points $x,y$
and then we place the edge $e=\{x,y\}$ in one of $k-2$ available edges
in the path $P_{k-2}$. Since there are two ways of placing the
edge $e$ in each slot of $P_{k-2}$
it follows that the number of paths from $u$ to $v$ that go through edge $e$ is equal to
$2(k-2)\cdot\binom{n-4}{k-3}\cdot (k-3)!$. If $k\leq (n-1)/2$ then the later quantity
is smaller than $\binom{n-3}{k-2}\cdot (k-2)!$, otherwise it is larger than $\binom{n-3}{k-2}\cdot (k-2)!$.
Therefore, if $k\leq (n-1)/2$, the fractional matching number of $\mathcal{H}$ is at least
$\frac{\binom{n-2}{k-1}\cdot (k-1)!}{\binom{n-3}{k-2}\cdot (k-2)!}=n-2$.
If  $k> (n-1)/2$ then the fractional matching number of $\mathcal{H}$ is at least $\frac{(n-2)(n-3)}{2(k-2)}$.
We have thus proven the following. \\

\begin{proposition} Let $G\in\mathcal{G}(n,p)$. Fix two vertices $u,v$ in $G$
and a positive integer $k$.  If $k\leq (n-1)/2$ then
\[ \mathbb{P}\left[\text{there is no path of length}\; k \; \text{between}\; u\;\text{and}\; v \right]   \leq
\left(1-p^k\right)^{n-2}  .\]
If $k>(n-1)/2$, then
\[ \mathbb{P}\left[\text{there is no path of length}\; k \; \text{between}\; u\;\text{and}\; v \right]   \leq
\left(1-p^k\right)^{\frac{(n-2)(n-3)}{2(k-2)}}  .\]
\end{proposition}

\subsubsection{Degrees}

Our paper ends with an estimate on the probability that
a $G\in \mathcal{G}(n,p)$ contains no vertex of fixed degree. \\

\begin{proposition} Let $G\in \mathcal{G}(n,p)$ and fix a positive integer $d\in \{0,1,\ldots,n-1\}$.
Then the probability that there is no vertex in $G$ whose degree equals $d$ is less than or equal to
\[\left(1-\binom{n-1}{d}p^d(1-p)^{n-1-d}\right)^{\frac{n}{2}}. \]
\end{proposition}
\begin{proof} This is yet another application of Theorem \ref{finnerbound} so we sketch it.
Let $v_1,\ldots,v_n$ be an enumeration of the vertices of $G$.
Let the hypergraph $\mathcal{H}=(\mathcal{V},\mathcal{E})$ be  defined as follows.
The vertex set $\mathcal{V}$
corresponds to the (potential) edges of $G$. The edge set $\mathcal{E}=\{E_1,\ldots,E_n\}$ corresponds to the vertices
of $G$. That is, for $i=1,\ldots,n$  the edge $E_i$ contains
those $u\in \mathcal{V}$ for which the corresponding edges of $G$ are incident to vertex $v_i$.
The result follows from the fact that
$|\mathcal{E}|=n$ and the maximum degree of $\mathcal{H}$ is equal to $2$.
\end{proof}

\section{Remarks}
As mentioned in Janson \cite{Jansonthree}, there exist collections of weakly dependent random variables  
that do not have a dependency graph. The dependencies between such collections of random variables 
can ocasionally be described using an independence system.
Recall that an \emph{independence system} is a pair 
$\mathcal{A} = (V, I)$ where $V$ is a finite set and $I$ is a collection 
of subsets of $V$ (called the \emph{independent sets}) with the following properties (see \cite{Bondy}): 
\begin{itemize}
 \item The empty set is independent, i.e., $\emptyset\in I$. 
 (Alternatively, at least one subset of $V$ is independent, i.e., $I \neq \emptyset$.)
 \item Every subset of an independent set is independent, i.e., for each $A'\subset A\subset \mathcal{A}$, if $A\in I$ then $A'\in I$. 
 This is sometimes called the \emph{hereditary property}.
\end{itemize} 

Given a set of random variables $\{Y_v\}_{v\in V}$, we say that their joint distribution 
is \emph{described with an independence system}, say $\mathcal{A}=(V,I),$ if  
for every $A\in I$ the random variables $\{Y_a\}_{a\in A}$ are mutually independent.  
Let us remark that this definition includes the case of $k$-wise independent random variables 
(see \cite{Alon}, Chapter $16$, or  \cite{Schmidt}). 
Notice that if $\{Y_v\}_{v\in V}$ are random variables whose joint 
distribution is described with an independence system 
$\mathcal{A}=(V,I)$ then $\{v\}\in I$, for all $v\in V$. 
It is easy to see that if the random variables $\{Y_v\}_{v\in V}$ have a dependency graph 
then their joint distribution is described with an independence system. 
However, the converse need not be true.    
In a similar way as in Section \ref{intro}, one may define the fractional chromatic number 
of an independene system as follows. \\
A $b$-fold coloring of an independence system $\mathcal{A} = (V, I)$ 
is a function $\lambda:I\rightarrow \mathbb{Z}_{+}$ 
such that $\sum_{A: v\in A} \lambda(A) = b,$ for all $v\in V$. 
The $b$-fold chromatic number of $\mathcal{A}$ is defined as $\chi_b(\mathcal{A}):= 
\inf_{\lambda}  \sum_{A \in I} \lambda(A),$ where the infimum is over all 
$b$-fold colorings, $\lambda(\cdot)$, of $\mathcal{A}=(V,I)$.
Finally, the fractional chromatic number of $\mathcal{A}$ is $\chi^*(\mathcal{A}):= \inf_{b}  \frac{\chi_b(\mathcal{A})}{b}.$

With these concepts by hand, one can  
prove a corresponding H\"older-type inequality 
using a similar argument as in Theorem \ref{graphFinner}. As a consequence 
one can obtain tail bounds similar to Theorem \ref{chromaticthm} and Theorem \ref{variancebound}, the only 
difference being that 
the fractional chromatic number of dependency graphs, $\chi^*(G)$, is  replaced with the fractional 
chromatic number of the independence system, $\chi^*(\mathcal{A})$. We leave the details to the reader.  



\textbf{Acknowledgements}
The authors are supported by ERC Starting Grant 240186 "MiGraNT, Mining Graphs and Networks: a Theory-based approach".

\end{document}